\documentclass[letterpaper, 10 pt, conference]{ieeeconf}  

\IEEEoverridecommandlockouts                              

\let\labelindent\relax
\usepackage{enumitem}
\usepackage{cite}

\PassOptionsToPackage{hyphens}{url}\usepackage{hyperref}

\usepackage{times} 
\usepackage{amsmath} 

\newtheorem{assum}{Assumption}
\newtheorem{definition}{Definition}
\newtheorem{lemma}{Lemma}
\newtheorem{thm}{Theorem}
\newtheorem{prop}{Proposition}
\newtheorem{cor}{Corollary}

\title{\LARGE \bf
Ride-Sharing Networks with Mixed Autonomy
}

\author{Qinshuang Wei, Jorge Alberto Rodriguez, Ramtin Pedarsani and Samuel Coogan
\thanks{Q. Wei, J. A. Rodriguez, and S. Coogan are with the School of Electrical and Computer Engineering, Georgia Institute of Technology, 
        \texttt{\small \{qinshuang, alberto.rod, sam.coogan\}@gatech.edu}. 
        S. Coogan is also with the School of Civil and Environmental Engineering, Georgia Institute of Technology. 
        R. Pedarsani is with the Department of Electrical Engineering, University of California, Santa Barbara, {\tt\small ramtin@ece.ucsb.edu}. This work was funded in part by the National Science Foundation under grants 1736582 and 1749357.}%
}

\begin{document}

\maketitle
\thispagestyle{empty}
\pagestyle{empty}

\begin{abstract}
We consider ride-sharing networks served by human-driven vehicles and autonomous vehicles. First, we propose a novel model for ride-sharing in this mixed autonomy setting for a multi-location network in which the platform sets prices for riders, compensation for drivers, and operates autonomous vehicles for a fixed price. Then we study the possible benefits, in the form of increased profits, to the ride-sharing platform that are possible by introducing autonomous vehicles. We first establish a nonconvex optimization problem characterizing the optimal profits for a network operating at a steady-state equilibrium and then propose a convex problem with the same optimal profits that allows for efficient computation. Next, we study the relative mix of autonomous and human-driven vehicles that results at equilibrium for various costs of operation for autonomous vehicles. In particular, we show that there is a regime for which the platform will choose to mix autonomous and human-driven vehicles in order to optimize profits. Our results  provide insights into how such ride-sharing platforms might choose to integrate autonomous vehicles into their fleet.

\end{abstract}

\section{Introduction}



Ride-sharing platforms, also known as transportation network companies, match passengers or riders with drivers using websites or mobile apps and have become ubiquitous in major cities across the world \cite{Mitchell2010, Board2016}. The rise of ride-sharing platforms coincides with a decline in private car ownership due to high costs, lack of parking, and persistent traffic congestion \cite{Hass-Klau2007, Mitchell2010, Javid2016, Mcbain2018}. 
Traditionally, rides are provided by drivers who use their own personal vehicle to provide service. 


However, ride-sharing platforms have indicated that they intend to incorporate autonomous vehicles (AVs) into their fleet in the near future \cite{Litman2017}. 
AVs managed by the ride-sharing platform can be directed to specific locations as needed. However, owning and managing an AV fleet can be costly for the ride-sharing platform, and significant technological hurdles remain before these platforms could transition to 100\% autonomous fleets. For these reasons, it is likely that ride-sharing platforms will, at least initially, adopt a hybrid or \emph{mixed} framework in which AVs operate along-side conventional, human-driven vehicles \cite{Boesch2016}. For example, the ride-sharing platform Lyft has partnered with Aptiv, a manufacturer of autonomous cars, to serve some of its ride requests in Las Vegas, Nevada with AVs \cite{Conger:2018ad}.  In this setting, the AVs can be deployed to serve, for instance, locations with abnormally high demand.

Research in ride-sharing platforms has largely focused on two ends of the autonomy spectrum. On one end, models of rider and driver behavior in conventional ride-sharing markets with no AVs have only appeared in the literature relatively recently \cite {Bimpikis2018, Banerjee2015, Cachon2017, Banerjee2016}. A common approach in these works is to conside ride-sharing as a two-sided market with passengers willing to pay for rides and drivers willing to provide rides for compensation.  On the other end, futuristic \emph{mobility-on-demand} systems consisting of only AVs have also been proposed and studied \cite{Zhang2015, Rigole2014, Zhang2015a, Zhang2016}. These works focus on controlling the movement of AVs to achieve objectives such as maximum throughput.

In this paper, we consider the transition from traditional ride-sharing networks to totally automated mobility-on-demand systems. We propose and study a model for  this mixed autonomy setting with both AVs and human-driven vehicles that is an extension of the model for traditional ride-sharing recently introduced in \cite{Bimpikis2018}. In the proposed model, the network consists of multiple equidistant locations, and at each time-step, potential riders arrive at these locations with desired destinations. The ride-sharing platform sets prices for riders and compensation to drivers in order to incentivize both riders and drivers to use the platform. In addition, the platform has the option to use AVs for a fixed cost. If human drivers and AVs are both present at a location when a rider requests a ride, in this work, we make the assumption that the platform exhausts the human drivers before using AVs. This assumption is motivated by the fact that, all else equal, the platform prefers to favor human drivers in order to encourage their continued engagement in the platform.

Here, as in \cite{Bimpikis2018}, we focus on the equilibrium conditions that arise in the resulting mixed autonomy network when the platform seeks to maximize profits, and we pose these conditions as a nonconvex optimization problem. We then provide an alternative convex problem from which an optimal solution to the original problem can be recovered. Next, we study the relative mix of autonomous and human-driven vehicles that results at equilibrium for various costs of operation for AVs. Our results support the intuition that, when the cost of operating AVs is high, the platform will not utilize them in its fleet, and when the cost is low, the platform will use only AVs. Perhaps surprisingly, we also show that there is a regime for which the platform will choose to mix autonomous and human-driven vehicles in order to maximize profits, thus providing insights into how such ride-sharing platforms might choose to integrate AVs into their fleet.

The remainder of this paper is organized as follows. Section \ref{sec:model} provides the model definition, and Section \ref{sec:optimization-problem} poses the problem of profit maximization as a nonconvex optimization problem. Section \ref{sec:alt-orig-optim} provides an alternative convex optimization problem that provides the same optimal profits. In Section \ref{sec:lower-bound-using}, we investigate a particular class of networks and provide a theoretical upper bound on the ratio of driver's compensation to the cost of AVs above which no AVs are used at the profit-maximizing equilibrium. Section \ref{sec:numerical-study} provides a numerical study further illuminating the implications of the proposed model, and concluding remarks are in Section \ref{sec:conclusions}.

\section{Problem Formulation}
\label{sec:model}
We consider an infinite horizon discrete time model of a ride-sharing network that extends the model recently proposed in \cite{Bimpikis2018} to accommodate a mixed autonomy setting with AVs and human-driven vehicles. The network is assumed to consist of  $n$ locations that have the same distance between each other. The network operator or \emph{platform} determines prices for rides and compensations to drivers within the network. We briefly summarize some of the key features and findings of the model introduced in \cite{Bimpikis2018} before providing the mathematical details of our extended model for mixed autonomy.



\begin{enumerate}
\item Drivers decide whether, when, and where to provide service so as to maximize their expected lifetime earnings. In particular, drivers have an outside option  with known earnings and only participate in the platform if compensations are such that in expectation their equilibrium lifetime earnings are at least this outside option.

\item 
The demand pattern of riders is stationary. As such, the analysis focuses on the equilibrium outcome determined by the platform's prices. This is reasonable for ride-sharing systems that involve a large number of drivers and riders.

\item The price of a ride may differ among locations, but does not depend on the desired destination of each rider. This is reasonable when all locations are equidistant, as is the case here.
\end{enumerate}

With these considerations in mind, we are interested in studying the potential benefits of adding AVs to the network to maximize the profit potential for the platform. 

\subsection{Model Definition}

We now formalize the mixed autonomous ride-sharing network described above. 

\noindent The model consists of the following parameters:
\begin{itemize}
\item $n$, number of equidistant locations

\item $\theta_i$, mass of potential riders that arrive at location $i$ in each period of time

\item $\alpha_{ij}$, fraction of riders at location $i$ who wish to go to location $j$ satisfying $\sum_{j} \alpha_{ij} = 1$ for all $i$. We assume $\alpha_{ii}=0$ for all $i$ and construct the $n$-by-$n$ adjacency matrix $\mathbf{A}$ as $[\mathbf{A}]_{ij} = \alpha_{ij}$ where $[\mathbf{A}]_{ij} $ denotes the $ij$-th entry of $\mathbf{A}$.


\item $(1-\beta)$, probability that a driver exits the platform after completing a ride, where  $\beta \in(0,1)$. Thus, drivers remain in the network and provide service for a finite amount of time, and a driver's expected lifetime in the network is $(1-\beta)^{-1}$.

\item $\omega$, outside option earnings for the driver's lifetime. Drivers only participate if the compensations $c_i$ provided by the platform are such that in expectation their equilibrium lifetime earnings are at least $\omega$.

\item $s$, the cost to the platform of operating an AV for the equivalent time of a driver's expected lifetime.


\item $k= \frac{s}{\omega}$, the ratio of the cost of operating an AV for the equivalent time of a driver's expected lifetime to the outside option earnings. 

\end{itemize}

With these parameters, we then introduce:
\begin{itemize}

\item $p_i$, the price set by the platform for a ride from location $i$.

\item $c_i$, the corresponding compensation for the driver for providing a ride at location $i$.

\item $F(\cdot)$, the continuous cumulative distribution of the riders' willingness to pay with the support $[0,\bar{p}]$. 
When confronted with a price $p$ for a ride, a fraction $1-F(p)$ of riders will accept this price, and the remaining $F(p)$ fraction will balk and leave the network without taking a ride. Note that $\theta_i (1-F(p_i ))$ is the effective demand for rides at location $i$.


\end{itemize}

In addition, we make the following assumption throughout the paper.

\begin{assum}
\label{assum:1}
\ 

\begin{itemize}
\item The network’s demand pattern is stationary with $\theta_i > 0$ for all $i\in\{1,\ldots, n\}$, i.e., $\left(\mathbf{A},\theta \right)$ is fixed for all time.
\item The directed graph defined by adjacency matrix $\mathbf{A}$ is strongly connected.
\end{itemize}
\end{assum}

In summary, the system consists of a \emph{platform} that sets prices, \emph{riders} that request rides among locations, \emph{drivers} who seek to maximize their compensation within the system, and \emph{AVs} 
managed by the platform alongside the drivers.


\subsection{Equilibrium Definition}
Next, we characterize the equilibrium conditions that  are induced by the stationary demand as characterized in Assumption \ref{assum:1} and by fixed prices and compensations set by the platform. An equilibrium for the system is a time-invariant distribution of the mass of riders, drivers, and AVs at each location, as formalized below.


Let $x_i$ denote the mass of drivers at location $i$. If there are fewer riders than drivers at a location, drivers can relocate to another location to provide service in the next time period. For each $j,k\in\{1,\ldots, n\}$, let $y_{jk}$ denote the drivers who do not get a ride at location $j$ and choose to relocate to location $k$ for the next period. It follows that
\begin{equation}
  \label{eq:1}
\sum_{k=1}^n y_{jk}= \max\left\{ x_j - \theta_j (1-F(p_j)), 0  \right\}  .
\end{equation}
Moreover, $ \sum_{j} y_{ji}$ is the mass of drivers who do not get a ride to any other location and choose to relocate to $i$. Further, let $\delta_i$ denote the mass of new drivers who choose to enter the platform and provide service at location $i$ at each time step. At equilibrium, it must hold that
\begin{equation}
  \label{eq:2}
x_i=\beta \left[\sum_{j=1}^n \alpha_{ji} \min \left\{ x_j , \theta_j(1-F(p_j)) \right\} + \sum_{j=1}^n y_{ji} \right] + \delta_i.
\end{equation}
In \eqref{eq:2}, observe that $\min \left\{ x_j , \theta_j(1-F(p_j)) \right\}$ is the total demand the platform serves with human-driven cars at location $j$, and therefore  $\sum_{j=1}^n \alpha_{ji} \min \left\{ x_j , \theta_j(1-F(p_j)) \right\}$ is the mass of drivers who find themselves located at $i$ after completing a ride. Recall that a fraction $\beta$ of drivers choose to stay in the network after each time step.


When the demand $\min \left\{ x_i, \theta_i(1-F(p_i)) \right\}$ at location $i$ exceeds the mass of available drivers $x_i$, the platform can choose to use AVs to meet this extra demand. Let $z_i$ denote the mass of AVs at location $i$, and for each $j,k\in\{1,\ldots,n\}$, let $r_{jk}$ denote the AVs which do not get a ride at $j$ and are relocated to location $k$. 
Then
\begin{align}
\nonumber   z_i=&\sum_{j=1}^n \alpha_{ji} \min \left\{ z_j , \max \left\{\theta_j(1-F(p_j))-x_j , 0\right\} \right\} \\
  \label{eq:5}&+ \sum_{j=1}^n r_{ji}.
\end{align}
In \eqref{eq:5}, observe that $\min \left\{ z_j , \max \left\{\theta_j(1-F(p_j))-x_j , 0\right\} \right\}$ is the total demand that the platform serves with AVs at location $j$ so that $\sum_{j=1}^n \alpha_{ji} \min \left\{ z_j , \max \left\{\theta_j(1-F(p_j))-x_j , 0\right\} \right\}$  is the mass of AVs which are located at $i$ after completing a ride. Moreover, $\sum_{j=1}^n r_{ji}$  is the mass of AVs which do not get a ride to any other location and are relocated to location $i$, and it follows that
\begin{equation}
\label{eq:6}
\sum_{k=1}^n r_{jk}= \max \left\{ z_j - \max \left\{ \theta_j (1-F(p_j))-x_j, 0  \right\}, 0 \right\} .
\end{equation}

Here, we assume that AVs are in continual use and do not leave the platform. 
Furthermore, as discussed in the Introduction, we assume human-drivers have priority in serving demand so that AVs are only used to meet demand that exceeds the mass of drivers. This is because, in a mixed autonomy setting, we assume the platform prefers to favor human drivers in order to encourage their continued engagement in the platform. 

Let $V_i$ denote the  expected earnings for a driver at location $i$ so that
\begin{align}
\nonumber V_i &= \min \left\{\frac{\theta_i (1-F(p_i))}{x_i}, 1 \right\}(c_i + \sum_{k=1}^n\alpha_{ik}\beta V_{k}) \\
 \label{eq:17} &+ \left(1-\min \left\{\frac{\theta_i (1-F(p_i))}{x_i}, 1 \right\}\right)\beta\max_j V_j.
\end{align}

Since drivers will only enter the platform if $V_i\geq \omega$, \emph{i.e.}, the expected earnings exceed the drivers' outside option, the platform will choose compensation such that $V_i=\omega$.

\begin{definition}
For some prices and compensations $\{p_i,c_i\}_{i=1}^n$, the collection $\left\{\delta_i,x_i,y_{ij},z_i,r_{ij} \right\}_{i,j=1}^n$ is \emph{an equilibrium under $\{p_i,c_i\}_{i=1}^n$} if \eqref{eq:1}--\eqref{eq:6} is satisfied and $V_i=\omega$ for all $i= 1, .., n$.
\end{definition}



\section{Profit-Maximization Optimization Problem}
\label{sec:optimization-problem}
We now consider the problem of maximizing profits at equilibrium. Maximizing the aggregate profit rate across the $n$ locations subject to the system’s equilibrium constraints yields the following optimization problem:
\begin{align}
\nonumber \max_{\{p_i,c_i\}_{i=1}^n}&\sum_{i=1}^n \left[ \min \left\{x_i+z_i, \theta_i (1-F(p_i)) \right\} \cdot p_i \right. \\
\nonumber &\left. - \min \left\{x_i,\theta_i (1-F(p_i)) \right\} \cdot c_i - z_i\cdot s\right]  \\
\nonumber \text{s.t.} &\left\{\delta_i,x_i,y_{ij},z_i,r_{ij} \right\}_{i,j=1}^n \text{ is an equilibrium}\\
  \label{eq:3} &\text{ under } \{p_i,c_i\}_{i=1}^n.
\end{align}
Next, we propose an equivalent optimization problem, 
followed by a lemma establishing the equivalence. 
To this end, consider 
\begin{align}
\nonumber \max_{\{p_i,\delta_i,x_i,y_{ij},z_i,r_{ij}\}}&\sum_{i=1}^n p_i \theta_i (1-F(p_i)) - \omega\sum_{i=1}^n \delta_i - s\sum_{i=1}^n z_i\\
\nonumber \text{s.t.}\quad d_i=&\theta_i (1-F(p_i))  \\
\nonumber  x_i=&\beta \left[\sum_{j=1}^n \alpha_{ji} \min \left\{ x_j , d_j \right\} + \sum_{j=1}^n y_{ji} \right] + \delta_i\\
\nonumber \sum_{j=1}^n y_{ij}=& \max\left\{ x_i - d_i, 0  \right\}\\
\nonumber  z_i=&\sum_{j=1}^n \alpha_{ji} \max \left\{d_j-x_j , 0\right\} + \sum_{j=1}^n r_{ji}\\
\nonumber \sum_{j=1}^n r_{ij}=&z_i -\max \left\{d_i-x_i , 0\right\}  \\
  \label{eq:4} p_i,&\delta_i,z_i,x_i,y_{ij},r_{ij}\geq 0 \qquad \forall i,j.
\end{align}

\begin{lemma} 
\label{lem:1}
Consider optimization problem \eqref{eq:4}. 
\begin{enumerate}
\item The optimal value computed from \eqref{eq:4} is an upper bound on the optimal profits computed via \eqref{eq:3}; thus it provides an upper bound on the profits generated by the platform with prices depending on the origin of a ride.

\item If $\left\{p_i, \delta_i, x_i, y_{ij}, z_i, r_{ij} \right\}_{i,j=1}^n$ is a feasible solution for \eqref{eq:4} such that $d_i > 0$ for all $i$, \emph{i.e.}, some riders are served at all locations, then there exist compensations $\left\{c_i \right\}_{i=1}^n$ such that the tuple $\left\{\delta_i, x_i, y_{ij}, z_i, r_{ij} \right\}_{i,j=1}^n$ constitutes an equilibrium under $\left\{p_i, c_i \right\}_{i=1}^n$. Furthermore, the cost incurred by the platform under these compensations per period is equal to $\omega\sum_{i=1}^n\delta_i$.

\item If, in addition, $(1-\beta)\omega< \bar{p}$, any optimal solution $\left\{p_i^*, \delta^*_i, x_i^*, y_{ij}^*, z_i^*, r_{ij}^* \right\}$ for \eqref{eq:4}, is such that $d_i^*>0$ for all $i$. Conversely, if $(1-\beta)\omega\geq \bar{p}$, any optimal solution for \eqref{eq:4} is such that $\delta_i^*=d_i^*=z_i^*=0$ for all $i$.
\end{enumerate}
\end{lemma}
\begin{proof}
The proof of the lemma closely follows that of \cite[Lemma 1]{Bimpikis2018}, where we adjust the claim and the proof so that it applies to the mixed-autonomy setting here. 

For the first part of the lemma, we need to show that any solution for \eqref{eq:3} satisfies $d_i=\theta_i (1-F(p_i)) \leq x_i+z_i$. By contradiction, suppose $d_i>x_i+z_i$, so that increasing the price $p_i$ by a small amount (and thus decreasing $\theta_i (1-F(p_i))$) will improve the value of the objective function. Therefore, $d_i\leq x_i+z_i$ at optimality. Hence we can write the first summation of \eqref{eq:3} as 
\begin{align}
  \label{eq:10}
  \sum_{i=1}^n \min \left\{x_i+z_i, \theta_i (1-F(p_i)) \right\} =\sum_{i=1}^n \theta_i (1-F(p_i)).
\end{align}
The term $\omega\sum_i\delta_i$ is the cost rate for drivers of the platform, which is a lower bound for the platform's cost on human-driven vehicles at equilibrium. Moreover, the constraints in \eqref{eq:4} correspond to the equilibrium constraints in \eqref{eq:3}. Therefore, the optimal value of \eqref{eq:4} is an upper bound for that of \eqref{eq:3}.

Next, we'll see that the upper bound can be reached by the optimal solution supported by some compensations $\left\{c_i \right\}_{i=1}^n$ under equilibrium.



To prove the second part of the lemma, we construct a compensation $\left\{c_i \right\}_{i=1}^n$ so that $V_i = \omega \quad\forall i$. To that end, let 
\begin{align}
  \label{eq:11}
  c_i &= \begin{cases}
                  \frac{x_i}{d_i}\cdot\omega(1-\beta) &\text{if }d_i<x_i\\
                  \omega(1-\beta) & \text{if } d_i\geq x_i.
\end{cases}
\end{align}

 Since we assumed that $d_i > 0$ for all $i$, then $c_i < \infty$ for all $i$ and thus the compensation is well-defined. Moreover, the probability that any driver at location $i$ is assigned to a ride is $\frac{d_i}{x_i}$ when $d_i< x_i$ and is 1 when $d_i \geq x_i$ since the driver takes the priority when drivers and AVs both exist in the platform. Therefore, the expected earnings for a single time period for a driver at location $i$ are equal to $\omega(1-\beta)$. Thus, the expected lifetime earnings are $V_i = \sum_j^\infty \beta^j\omega(1-\beta) = \omega$. Hence, the solution $\left\{p_i, \delta_i, x_i, y_{ij}, z_i, r_{ij} \right\}_{i,j=1}^n$ is supported as an equilibrium using the compensations we constructed above.

Moreover, the cost incurred by the platform under these compensations per period is
\begin{equation*}
\sum_{i=1}^n \min \left\{x_i,\theta_i (1-F(p_i)) \right\}\cdot c_i = \sum_{i=1}^n \min \left\{x_i,d_i \right\} \cdot c_i
\end{equation*}
We construct a partition for the locations so that $I_1 = \left\{i: d_i<x_i \right\}$ and $I_2 = \left\{i: d_i\geq x_i \right\}$. Therefore $\sum_i \min \left\{x_i,d_i \right\} \cdot c_i = \sum_{i\in I_1} d_i c_i + \sum_{i\in I_2} x_i c_i = \sum_{i\in I_1} d_i\cdot \frac{x_i}{d_i}\omega(1-\beta) + \sum_{i\in I_2} x_i \omega(1-\beta) = \sum_i x_i\omega(1-\beta) = \sum_{i=1}^n \delta_i\omega$. The last equality follows from the fact that $\sum_{i=1}^n x_i(1-\beta) = \sum_{i=1}^n\delta_i$ since, at equilibrium, the mass drivers entering the platform is equal to the mass of drivers that are leaving.

The third part of the lemma follows directly from the second part of \cite[Lemma 1]{Bimpikis2018} since $z_i>0$ only if $d_i>0$ in our scenario.
\end{proof}

From Lemma \ref{lem:1}, we conclude that it is without loss of generality for us to focus on the optimization problem \eqref{eq:4} for the rest of the paper.

Moreover, while the objective function of \eqref{eq:4} is not concave in general, it is concave for distributions for which the first summation $\sum_{i=1}^n p_i\theta_i(1-F(p_i))$, the revenue of the platform, is concave. This is true, for example, for the case that $F(\cdot)$ is the uniform distribution and some other distributions. Throughout the rest of the paper, we focus on the case where the rider's willingness to pay is such that the objective function of \eqref{eq:4} is concave.  

\section{Alternative Optimization Problem}
\label{sec:alt-orig-optim}
Even when \eqref{eq:4} possesses a concave objective function, the constraints are non-convex so that solving \eqref{eq:4} remains difficult. This section introduces an alternative optimization problem of the mixed-autonomy model for which the optimal profits will be the same as that of \eqref{eq:4}. 
While the optimal profits are the same, the optimal solutions of these two problems are not exactly the same. However, given the optimal solution of one of the optimization problems, we show that we are able to compute an optimal solution for the other problem with identical profit. In particular, the alternative optimization problem becomes a quadratic optimization problem with linear constraints when $F(\cdot)$ is a uniform distribution.

Consider the optimization problem
\begin{align}
\nonumber \max_{\{p_i,\delta_i,x_i,z_i,r_{ij}\}}&\sum_{i=1}^n p_i \theta_i (1-F(p_i)) - \omega\sum_{i=1}^n \delta_i - s\sum_{i=1}^n z_i\\
\nonumber  \text{s.t.}\quad d_i&=\theta_i (1-F(p_i))\\
\nonumber  x_i&=\beta \sum_{j=1}^n \alpha_{ji}x_j + \delta_i\\
\nonumber  z_i&=\sum_{j=1}^n \alpha_{ji} (d_j-x_j)+ \sum_{j=1}^n r_{ji}\\
\nonumber \sum_{j=1}^n r_{ij}&=z_i -(d_i-x_i )\\
  \label{eq:7} p_i,\delta_i,x_i,z_i,r_{ij}&\geq 0 \qquad \forall i,j.
\end{align}

In the following, we regard \eqref{eq:4} as the \emph{original} optimization problem and \eqref{eq:7} as the \emph{alternative} optimization problem. 

Theorem 1 below states that \eqref{eq:4} and \eqref{eq:7} have the same optimal profits for any $\beta$, $k$ and adjacency matrix $\mathbf{A}$. Moreover, given one optimal solution for \eqref{eq:4} or \eqref{eq:7}, it is possible to compute an optimal solution for the other.

\begin{thm}
\label{thm:equiv}
Consider the original optimization problem \eqref{eq:4} and alternative optimization problem \eqref{eq:7}. Let
\begin{align}
  \label{eq:18}
  \mathbf{u}^{ori*} = \left\{p_i^{ori*}, \delta_i^{ori*}, z_i^{ori*}, x_i^{ori*}, y_{ij}^{ori*}, r_{ij}^{ori*}\right\}_{i,j=1}^n
\end{align}
be an optimal solution for \eqref{eq:4} and 
\begin{align}
  \label{eq:19}
  \mathbf{u}^{alt*} = \left\{p_i^{alt*}, \delta_i^{alt*}, z_i^{alt*}, x_i^{alt*}, r_{ij}^{alt*}\right\}_{i,j=1}^n
\end{align}
be an optimal solution for \eqref{eq:7}. 
Then the following holds:
\begin{itemize}
\item The original optimization problem and the alternative problem obtain the same optimal profits for all possible choices of $\beta$, $k$ and adjacency matrix $\mathbf{A}$.

\item The optimal solutions satisfy $x^{ori*} =  x^{alt*}$, $z^{ori*} = z^{alt*}$, $p^{ori*} = p^{alt*}$ and $\delta^{ori*} = \delta^{alt*}$.

\item If $\theta_i (1-F(p^{ori*}_i))\leq x_i^{ori*}$ for all $i$ in the original optimization problem, then $z_i^{ori*}=0$ for all $i$ and setting $r^{alt*}_{ij} = y^{ori*}_{ij}$ for all $i,j$ constitutes an optimal solution for the alternative problem. 

\item If $ \theta_i (1-F(p^{alt*}_i))\leq x_i^{alt*}$ for all $i$ in the alternative optimization problem, then $z_i^{alt*}=0$ for all $i$ and setting $y^{ori*}_{ij} = r^{alt*}_{ij}$ constitutes an optimal solution for the original optimization problem. 
\end{itemize}
\end{thm}
\newcommand{\profit}{\phi}
\begin{proof}
Let $\profit^{ori*}$ and $\profit^{alt*}$ represent the optimal profits of the two problems \eqref{eq:4} and \eqref{eq:7}, respectively, and let $d^{ori*}_i=\theta_i (1-F(p^{ori*}_i))$ and $d^{alt*}_i= \theta_i (1-F(p^{alt*}_i))$.

To prove that the optimal profits of the two problems are equal, we first show that $\profit^{ori*} \leq  \profit^{alt*}$ and then $\profit^{ori*} \geq  \profit^{alt*}$.

We consider three cases to prove $\profit^{ori*} \leq  \profit^{alt*}$.

\underline{Case 1:} $d^{ori*}_i\geq x^{ori*}_i \quad \forall i$. Then $\mathbf{u}^{ori*}$ will be feasible for alternative problem because they lead to the same optimization problem in this case. Therefore $\profit^{ori*} \leq  \profit^{alt*}$.

\underline{Case 2:} $d^{ori*}_i\leq x^{ori*}_i \quad \forall i$. Then the AVs are not needed in any location and $z_i = 0, r_{ij} = 0 \quad \forall i,j$. Then the original optimization problem becomes

\begin{align}
\nonumber \max_{\{p_i,\delta_i,x_i,y_{ij},z_i,r_{ij}\}} \qquad &\hspace{-5pt}\sum_{i=1}^n p_i \theta_i (1-F(p_i)) - \omega\sum_{i=1}^n \delta_i\\
\nonumber  \text{s.t.}\quad d_i&=\theta_i (1-F(p_i))\\
\nonumber  x_i&=\beta \left[\sum_{j=1}^n \alpha_{ji} d_j + \sum_{j=1}^n y_{ji} \right] + \delta_i\\
\nonumber \sum_{j=1}^n y_{ij}&= x_i - d_i\\
  \label{eq:8} p_i,\delta_i,x_i,y_{ij}&\geq 0 \qquad \forall i,j.
\end{align}

Let $z^{alt}_i=0$ and $y^{alt}_{ij}=0 \quad \forall i,j$. Then the alternative problem becomes exactly the same problem as \eqref{eq:8} when we substitute $r_{ij}$ with $y_{ij}$, which proves the claim.



\underline{Case 3:} There exists some location $i$ such that $x^{ori*}_i > d^{ori*}_i$ and some location $j$ such that $x^{ori*}_j < d^{ori*}_j$. In this case, if there is no $i$ such that  $x^{ori*}_i = d^{ori*}_i$, then let $I_1=\{i: x^{ori*}_i > d^{ori*}_i \}$ and let $I_2=\{i: x^{ori*}_i < d^{ori*}_i \}$. We can then consider an aggregated network with locations $1$ and $2$ representing the combined locations in $I_1$ and $I_2$, respectively. 
Hence, in this aggregated network, $\alpha_{11} = \alpha_{22} \geq 0$; $\alpha_{12} >0$ and $\alpha_{21} >0$ by our assumption that the directed graph defined by adjacency matrix A is strongly connected. 
Since $d^{ori*}_1<x^{ori*}_1$, then $z^{ori*}_1 = 0$. But $z^{ori*}_1 = \max \left\{d^{ori*}_1-x^{ori*}_1 , 0\right\} \alpha_{11} + \max \left\{d^{ori*}_2-x^{ori*}_2 , 0\right\} \alpha_{21} + \sum_{j=1}^n r^{ori*}_{j1} = (d^{ori*}_2-x^{ori*}_2)\alpha_{21}+\sum_{j=1}^n r^{ori*}_{j1}$ since $d^{ori*}_2-x^{ori*}_2 >0$ and $d^{ori*}_1-x^{ori*}_1 <0$. Hence $z_1^{ori*}>0$ which leads to a contradiction. Therefore if there is no $i$ such that  $x^{ori*}_i = d^{ori*}_i$, then either $x^{ori*}_i > d^{ori*}_i \quad$ for all $i$ or $x^{ori*}_i < d^{ori*}_i$  for all $i$.

If there exists $i$ such that $x^{ori*}_i = d^{ori*}_i$, define $I_1$ and $I_2$ as above and introduce $I_3=\{i: x^{ori*}_i = d^{ori*}_i\}$. 
Similar to the above argument, we show that $z^{ori*}_1 =z^{ori*}_3 = 0$. Since $z^{ori*}_1=\sum_{j=1}^n \alpha_{j1} \max \left\{d^{ori*}_j-x^{ori*}_j , 0\right\} + \sum_{j=1}^n r^{ori*}_{j1}$ while $d^{ori*}_2-x^{ori*}_2>0$, then $\alpha_{21}=0$. Similarly, we must have $\alpha_{23}=0$. Therefore, we have $\alpha_{22}=1$ since $\sum_{j=1}^3 \alpha_{ij}=1$. However, $\alpha_{22}=1$ means that some components in the graph are not strongly connected with the others, which contradicts our assumption. Hence this mixed situation cannot be an optimal solution for the problem.
Thus, up to now, we have shown that $\profit^{ori*} \leq  \profit^{alt*}$.
Next we show that $\profit^{ori*} \geq  \profit^{alt*}$. 


\underline{Case 1:} If $d^{alt*}_i\geq x^{alt*}_i$ for all $i$, then $\mathbf{u}^{alt*}$ will be feasible for the original problem because they lead to the same problem in this case. Therefore $\profit^{ori*} \geq  \profit^{alt*}$.

\underline{Case 2:} If $d^{alt*}_i\leq x^{alt*}_i \quad \forall i$, suppose $\exists i$ s.t. $z^{alt*}_i >0$. Then $\sum_{i=1}^n {z^{alt*}_i}>0$. Recall
 $z_i = \sum_{j=1}^n r_{ij}+(d_i-x_i ) = \sum_{j=1}^n \alpha_{ji} (d_j-x_j)+ \sum_{j=1}^n r_{ji}$ in the alternative problem, thus $\exists j,k$ s.t. $r^{alt*}_{ij} >0$ and $r^{alt*}_{ki} >0$. Since $r_{ii}$ is not affected by any other equations, if both $i^{th}$ row and $i^{th}$ column are greater than 0, we produce the same result by adding a small amount  $\epsilon$ to $r^{alt*}_{ii}$, reduce the same amount from $r^{alt*}_{ij}$ and  $r^{alt*}_{ki}$ and add  $\epsilon$ to $r^{alt*}_{kj}$. After this transformation, the other variables are not changed and the new $r$ will produce the same profit as before while $r_{ii}>0$. However, from KKT conditions \eqref{eq:12-3} and  \eqref{eq:12-4} that are discussed in detail in Section \ref{sec:lower-bound-using}, if $z_i>0$, then $r_{ii}>0$ will not produce an optimal result (notice that \eqref{eq:12-3} and  \eqref{eq:12-4} hold at optimality even without Assumption \ref{assum:2}). Therefore $z_i=0 \quad \forall i$ in this case. 
Hence $\mathbf{u}^{ori}$ = 
$\left\{p_i^{ori}=p_i^{alt*}, \delta_i^{ori}=\delta_i^{alt*}, z_i^{ori}=0,x_i^{ori}= \right. \\ 
\left. x_i^{alt*}, y_{ij}^{ori}=r_{ij}^{alt*}, r_{ij}^{ori}=0 \right\}_{i,j=1}^n $
will produce the same profit in the original optimization problem by the same justification as in case 2 of this proof. Thus $\profit^{ori*} \geq  \profit^{alt*}$.

\underline{Case 3:} If there exist $\beta$ and k such that the optimal solution $\mathbf{u}^{alt*}$ does not satisfy the two situations above, which means there exist locations such that $d^{alt*}_i > x^{alt*}_i$ for some i and $d^{alt*}_j < x^{alt*}_j$ for some j at the same time. Also there may exist some locations k such that $d^{alt*}_k = x^{alt*}_k$.
By combining the locations having the same relation between $\mathbf{d}$ and $\mathbf{x}$, we can reduce the problem to a three location problem where $x^{alt*}_1>d^{alt*}_1$, $x^{alt*}_2<d^{alt*}_2$ and $x^{alt*}_3=d^{alt*}_3$.
Therefore, in Table \ref{tab:1}, we list and compare the constraints between the original optimization problem and the alternative one (knowing the objective functions are the same for both cases).

Let $\left\{z_i^{ori} \right\}_{i=1}^n= \left\{z_i^{alt*} \right\}_{i=1}^n$, $\left\{x_i^{ori} \right\}_{i=1}^n = \left\{x_i^{alt*} \right\}_{i=1}^n$,  $\left\{\delta_i^{ori} \right\}_{i=1}^n = \left\{\delta_i^{alt*} \right\}_{i=1}^n$ and $\left\{p_i^{ori} \right\}_{i=1}^n = \left\{p^{alt*} \right\}_{i=1}^n$ (which indicates $\left\{d_i^{ori} \right\}_{i=1}^n = \left\{d_i^{alt*} \right\}_{i=1}^n$). This guarantees the profit of the original problem to be the same as the optimal profit of the alternative problem. Moreover, there exist $\left\{y_{ij}^{ori} \right\}_{i,j=1}^n$ and $\left\{r_{ij}^{ori} \right\}_{i,j=1}^n$ so that $\mathbf{u}^{ori}$ will be a feasible solution for the original problem. To keep the succinctness of our formula, let $\Delta = (x_1-d_1)>0$.

Let $\mathbf{y}^{ori}$ =$\begin{bmatrix}
    \alpha_{11}\Delta & \alpha_{12}\Delta & \alpha_{13}\Delta\\
    0 & 0 & 0\\
    0 & 0 & 0
\end{bmatrix}$ 
and $\mathbf{r}^{ori}$ =
$\begin{bmatrix}
    r^{alt*}_{11}-\alpha_{11}\Delta & r^{alt*}_{12}-\alpha_{12}\Delta &r^{alt*}_{13} - \alpha_{13}\Delta\\	
    r^{alt*}_{21} & r^{alt*}_{22} & r^{alt*}_{23}\\
    r^{alt*}_{31} & r^{alt*}_{32} & r^{alt*}_{33}
\end{bmatrix}$. 

Then the constraints for the original problem are all satisfied while the profit remains unchanged. Therefore $\profit^{ori*} \geq  \profit^{alt*}$.

Similarly, if there's no location that $x^{alt*}_i=d^{alt*}_i$, then by setting $\alpha_{i3}=\alpha_{3i}=0$ for all $i$, we will obtain the same result.

\begin{table}[ht]
\caption{The List of Constraints}
\label{constraints}
\begin{center}
\setlength\tabcolsep{1pt}
\begin{tabular}{|r l|r l|}
    \hline
    \multicolumn{2}{|c|}{Original} &     \multicolumn{2}{|c|}{Alternative} \\ 
    \hline 
    $x_1 $&$=  \delta_1 + \beta \Big[\alpha_{21} x_2 +\alpha_{31} x_3$  &    $x_1 $&$= \delta_1 + \beta \Big[\alpha_{21} x_2 +\alpha_{31} x_3$\\
& \quad   $ + \alpha_{11} d_1 + \sum_j y_{j1} \Big]$ &   & \quad $+ \alpha_{11}x_1$\Big]\\
    $\sum_j y_{1j}$&$= x_1 - d_1$ &  &\\
    $z_1 $&$= \alpha_{21}(d_2-x_2)+\sum_j r_{j1}$ &  $z_1 $&$= \alpha_{21}(d_2-x_2)+\sum_j r_{j1}$\\
& & & $\quad +\alpha_{11}(d_1-x_1)$ \\
    $\sum_j r_{1j} $&$= z_1$ & $\sum_j r_{1j} $&$= z_1-(d_1-x_1)$\\
    \hline
    $x_2 $&$= \delta_2 + \beta \Big[\alpha_{12} d_1 +\alpha_{22} x_2 $ & $x_2 $&$= \delta_2 + \beta \Big[\alpha_{12} x_1 +\alpha_{22} x_2 $\\
&  \quad   $+ \alpha_{32}x_3 + \sum_j y_{j2} \Big]$ &  & $\quad +\alpha_{32}x_3 \Big]$\\
    
    $\sum_j y_{2j}$&$= 0$ &   &\\
    $z_2 $&$= \alpha_{22}(d_2-x_2) + \sum_j r_{j2}$ & $z_2 $&$= \alpha_{22}(d_2-x_2)+\sum_j r_{j2}$\\
& & & $\quad +\alpha_{12}(d_1-x_1)$ \\
    $\sum_j r_{2j} $&$= z_2-(d_2-x_2)$ & $\sum_j r_{2j} $&$= z_2-(d_2-x_2)$\\
    \hline    
    $x_3 $&$= \delta_3 + \beta \Big[\alpha_{13} d_1 +\alpha_{23} x_2 $ & $x_3 $&$= \delta_3 + \beta \Big[\alpha_{13} x_1 +\alpha_{23} x_2 $\\
& \quad    $ +\alpha_{33}x_3 + \sum_j y_{j3} \Big] $ & & $ \quad +\alpha_{33}x_3$\Big]\\
    $\sum_j y_{3j}$&$= 0$ &  & \\
    $z_3 $&$=  \alpha_{23}(d_2-x_2)$ & $z_3 $&$= \alpha_{23}(d_2-x_2)$\\
  &\quad    $ + \sum_j r_{j3}$ &&  $+\alpha_{13}(d_1-x_1) + \sum_j r_{j3}$\\
    $\sum_j r_{3j} $&$= z_3$ & $\sum_j r_{3j} $&$= z_3$\\
    \hline    
        \multicolumn{2}{|c|}{$p_i,\delta_i,x_i,y_{ij},z_i,r_{ij}\geq 0 \quad \forall i,j$} &     \multicolumn{2}{|c|}{$p_i,\delta_i,x_i,y_{ij},z_i,r_{ij}\geq 0 \quad \forall i,j$}\\
     \hline
\end{tabular}
\end{center}
\label{tab:1}
\end{table}

Moreover, since we know that this $\mathbf{u}^{ori}$ will fall into the third case and is not optimal for the original problem, while $\profit^{ori*} \leq  \profit^{alt*}$, then case 3 is not optimal for the alternative problem. 

Therefore, $\profit^{ori*} \geq  \profit^{alt*}$  and thus $\profit^{ori*} = \profit^{alt*}$.
\end{proof}

\begin{definition}
We refer to the original problem \eqref{eq:4} or its alternation \eqref{eq:7} as the \emph{mixed autonomy system}. We refer to \eqref{eq:4} with the additional constraint that $z_i=0$ for all $i$ as the \emph{human-only} system.
\end{definition}

Corollary \ref{cor:profit} below can be derived from part of the proof process above.

\begin{cor}
\label{cor:profit}
  The optimal profit for the mixed-autonomy network will be no less than the one for the human-only network, \emph{i.e.}, a network for which $z_i=0$ for all $i$.
\end{cor}
\begin{proof}
  The mixed-autonomy optimization problem can be transformed into \eqref{eq:8} by setting $\mathbf{z=0}$ and $\mathbf{r=0}$. Furthermore, \eqref{eq:8} is exactly the optimization problem for the system without any AVs. Therefore, by letting $\mathbf{z=0}$ and $\mathbf{r=0}$ and the other variables equal to the optimal solution for the optimization problem for the system without AV, we obtain a feasible solution for the mixed-autonomy system. Therefore the optimal profit for the mixed-autonomy system will be no less than that of the system without autonomous system.
\end{proof}

\section{When Are AVs Beneficial?} 
\label{sec:lower-bound-using}
We now turn our attention to studying properties of the equilibria that result from the model and optimization problems proposed above. We particularly seek to understand when the platform stands to benefit from introducing AVs in its fleets. We thus are especially interested in comparing the solution to the above profit-maximizing optimization problem with mixed autonomy to the solution with no AVs. To that end, we have the following.


 In this section we first define a family of \emph{star-to-complete networks} introduced in \cite{Bimpikis2018}. Then we establish an upper bound on $k$ for which the profit of the network with AVs will be greater than the human-only network only if $k$ is greater than this bound.
We first compute the Lagrangian of the alternative optimization problem \eqref{eq:7} since it has the same optimal profits as the mixed-autonomy network. Then, a lemma establishes when the optimal profits of the mixed-autonomy system and the human-only will diverge, followed by a proposition that gives an upper bound of $k$ for the diversion.

\begin{definition}
The class of demand patterns $(\mathbf{A}^\xi,\mathbf{1})$ with $n\geq 3$, $\xi \in \left[0,1 \right]$, and
\begin{align}
\mathbf{A}^\xi &= \begin{bmatrix}
    0 & \frac{1}{n-1} & \frac{1}{n-1} & \dots & \frac{1}{n-1} \\
    c_1 & 0 & c_2 & \dots & c_2 \\   
    c_1 & c_2 & 0 & \dots & c_2 \\
    \vdots &  \vdots &  \vdots &  \ddots &  \vdots\\
    c_1 & c_2 & \dots & c_2 & 0 \\
\end{bmatrix},\\
c_1 &= \frac{\xi}{n-1}+(1-\xi),\qquad
c_2 = \frac{\xi}{n-1}
\end{align}
is the family of \emph{star-to-complete} networks. It is a \emph{star} network when $\xi=0$ for which we write $\mathbf{A}^S=\mathbf{A}^0$ and a \emph{complete} network when $\xi=1$ for which we write $\mathbf{A}^C=\mathbf{A}^1$. Therefore the general adjacency matrix of a \emph{star-to-complete} network can be written as $\mathbf{A}^\xi$ = $\xi \mathbf{A}^{C} + (1-\xi) \mathbf{A}^{S}$.
\end{definition}

In addition, we make the following assumption throughout the rest of paper.
\begin{assum}
\label{assum:2}
  All locations have the same mass of potential riders, which we normalize to one, i.e., $\mathbf{\theta}=\mathbf{1}$. Also, the rider's willingness to pay is uniformly distributed in $\left[0,1 \right]$ so that $F(p)=p$ for $p\in[0,1]$.
\end{assum}

Define
\begin{multline*}
L(\mathbf{p,r,x},\boldsymbol{\delta,\lambda})=\mathbf{p}^T\mathbf{(1-p)-}\omega\mathbf{1}^T\boldsymbol{\delta}-s\mathbf{1}^T\mathbf{z}\\
+\boldsymbol{\lambda}^T\left[ \mathbf{A}^T\mathbf{(1-p-x)+r}^T\mathbf{1-z} \right]\\
+\boldsymbol{\gamma}^T\left[\mathbf{1-p-x+r1-z}\right]+ \boldsymbol{\mu}^T\left[\beta \mathbf{A}^T\mathbf{x}+\boldsymbol{\delta}-\mathbf{x}\right].
\end{multline*}
That is, $L$ is the \emph{Lagrangian} function associated with the optimization problem \eqref{eq:7} under Assumption \ref{assum:2}.

By the KKT conditions\cite{Boyd2004}, we obtain the following properties that must hold at optimality:
\begin{align}
  \label{eq:12}
\frac{\partial L}{\partial \delta_{i}} &= -\omega +\mu_i \leq 0\\ 
  \label{eq:12-2}\frac{\partial L}{\partial x_{i}} &= \sum_{j=1}^n (\beta \mu_j -\lambda_j)\alpha_{ij}-\gamma_i-\mu_i\leq 0 \\
  \label{eq:12-3}\frac{\partial L}{\partial z_{i}}& = -s - \lambda_i- \gamma_i \leq 0\\
  \label{eq:12-4}\frac{\partial L}{\partial r_{ij}} &=  \lambda_j + \gamma_i \leq 0. 
\end{align}

If $\delta_{i}$, $x_{i}$, $z_{i}$ or $r_{ij}$ is strictly greater than zero, then the equality of the corresponding differential equation must hold.

The next Lemma establishes that, for a star-to-complete network, if it is optimal for the platform to use AVs at some location, then it is optimal to use AVs at all locations.
\begin{lemma}
\label{lem:aut}
  If, in some star-to-complete network, the optimal profit of the mixed-autonomy system is strictly greater than that of the human-only system, then the optimal solution of the mixed-autonomy system satisfies $z_i>0$ for all $i$.
\end{lemma}

\begin{proof}
Recall the adjacency matrix $\mathbf{A}^\xi$ for a star-to-complete network. 
Assume the optimal profit of the mixed-autonomy system is greater than that of the original system. Then there must exist $k$ such that $z_k>0$ and hence $d_k-x_k>0$. Moreover, for any $i\neq k$, $z_i=\sum_{j=1}^n  \alpha_{ji} \max \left\{d_j-x_j , 0\right\} + \sum_{j=1}^n  r_{ji} = \sum_{j\neq k} \alpha_{ji} \max \left\{d_j-x_j , 0\right\}+ \alpha_{ki}(d_k-x_k) + \sum_{j=1}^n r_{ji} > \alpha_{ki}(d_k-x_k)>0$ because $\alpha_{ki}>0$ if $\xi>0$. Therefore, if $\xi>0$, then $z_i>0\quad$ for all $i$.

Further, if $\xi=0$, then $\alpha_{1j} = \frac{1}{n-1}$ and $\alpha_{j1} =1\quad\forall j>1$ while all of the other elements in $A^\xi$ are 0. If $k=1$, then $z_i = \frac{1}{n-1}(d_1-x_1)+ \sum_{j=1}^n  r_{ji} > 0$ for all $i\neq 1$. If $k>1$, then $z_1 = \sum_{j=1}^n \max \left\{d_j-x_j , 0\right\} + \sum_{j=1}^n  r_{j1} = \sum_{j\neq 1,k}\max \left\{d_j-x_j , 0\right\} + (d_k-x_k) +\sum_{j=1}^n  r_{j1}>0$ and thus $z_i >0$ for all $ i$.

Therefore the optimal profit of the mixed-autonomy system is greater than that of the original system only if $z_i >0$ for all $i$.
\end{proof}

Next, we establish a necessary condition required for the platform to find it beneficial to introduce AVs in the network.
\begin{prop}
\label{prop:lowerbound}
  Consider a star-to-complete network, if the optimal profit of the mixed-autonomy system is strictly greater than that of the human-only system, then $k \leq 1-\beta$.
\end{prop}
\begin{proof}
The optimal profit of the mixed autonomy system will always be no less than that of the human-only system. Therefore, if there is a difference between the two optimal profits, then the optimal profit of the mixed-autonomy system will be strictly greater than that of the  human-only system. Hence by Lemma \ref{lem:aut}, $z_i >0$ for all $i$ under the situation.

Therefore $\lambda_i+\gamma_i=-s$ for all $i$. From \eqref{eq:12-2}, we have
\begin{align}
  \label{eq:14}
\frac{1}{n-1} \beta \sum_{j=2}^{n}\mu_j-\frac{1}{n-1}\sum_{j=2}^{n}\lambda_j-\gamma_1-\mu_1 &\leq 0,
\end{align}
and, for all $i\neq 1$,
\begin{align}
  \label{eq:14-2} c_1(\beta\mu_1-\lambda_1)+c_2 \sum_{j\neq 1,i}^{n}(\beta\mu_j-\lambda_j)-\gamma_i - \mu_i &\leq 0. 
\end{align}

Therefore combining \eqref{eq:14} and \eqref{eq:14-2} gives
\begin{align*}
&(n-1)c_1\left(\frac{1}{n-1} \beta \sum_{j=2}^{n}\mu_j-\frac{1}{n-1}\sum_{j=2}^{n}\lambda_j-\gamma_1-\mu_1\right)\\ 
&+ \sum_{i=1}^{n} \left[ c_1(\beta\mu_1-\lambda_1)+c_2 \sum_{j\neq 1,i}^{n}(\beta\mu_j-\lambda_j)-\gamma_i - \mu_i   \right] \leq 0
\end{align*}
and after recombination of the variables, we have
\begin{align*}
&-c_1(n-1)(\lambda_1+\gamma_1)-\left[c_1+(n-2)c_2 \right]\sum_{j=2}^{n}\lambda_j - \sum_{j=2}^{n}\gamma_j \\
&+ c_1(n-1)(\beta-1)\mu_1+\left[c_1\beta+(n-2)c_2\beta-1 \right]\sum_{j=2}^{n}\mu_j \\
&\leq 0.
\end{align*}
Knowing that $c_1+(n-2)c_2 = 1$ and $\mu_{i}\leq \omega$ for all $i$ from \eqref{eq:12}, we then have 
\begin{align*}
c_1(n-1)s+&(n-1)s \\
&\leq  -c_1(n-1)(\beta-1)\mu_1-(\beta-1)\sum_{j=2}^{n}\mu_j \\
&\leq (c_1+1)(n-1)(1-\beta)\omega\\ 
\implies \frac{s}{\omega} &\leq \frac{(c_1+1)(n-1)(1-\beta)}{(c_1+1)(n-1)} = 1-\beta.
\end{align*}

Notice that the fraction $\frac{s}{\omega}$ is $k$. Therefore, for all star-to-complete networks, the profit of the mixed-autonomy system will be strictly greater than that of the human-only system only if $k \leq 1-\beta$.
\end{proof}

Proposition \ref{prop:lowerbound} establishes that $k \leq 1-\beta$ is a necessary condition for the platform to decide to introduce AVs into its fleet, however, it is not a sufficient condition. In the numerical study of the next section, we observe that the platform will find it beneficial to have no AVs, some AVs, or all AVs for certain regimes of $k$ for each fixed $\beta$.
\section{Numerical study}
\label{sec:numerical-study}
We now provide a case study that illustrates the results developed above and illuminates the potential benefits of mixed-autonomy in ride-sharing.

We consider a simple 3-location star-network satisfying Assumption \ref{assum:2} with location $1$ connected to locations $2$ and $3$, and locations $2$ and $3$ are only connected to location $1$. We see below that even this simple network exhibits interesting properties in the mixed autonomy setting. 
We take the adjacency matrix to be
\begin{align}
  \label{eq:16}
A = \begin{bmatrix}
0 & \frac{1}{2} & \frac{1}{2}\\
1 & 0 & 0 \\
1 & 0 & 0 
\end{bmatrix}
\end{align}
so that riders originating at location $1$ have locations $2$ and $3$ as their destination with equal probability, while riders at locations $2$ and $3$ are all traveling to location $1$; for example, location $1$ is a city center.
 Furthermore, we assume without loss of generality that the driver's outside option is $\omega=1$.  The alternative optimization problem \ref{eq:7}  and the restriction to only human drivers \ref{eq:8} are implemented in CVX, a Matlab-based convex optimization software package.
 
We compute the optimal profits and resulting equilibria for a range of values of $\beta$ and $k$.
In Table \ref{tab:tab1}, the first column lists several values of $\beta$. The fourth column gives $k_t:=1-\beta$, the corresponding theoretical upper bound established in Proposition \ref{prop:lowerbound} such that for $k\geq k_t$, it is optimal for the platform to use humans only. Next, we introduce two particular values of $k$ obtained numerically for which the optimal solution changes qualitatively. The number in the second column, labeled $k_a$, is the value below which no human drivers are present at optimality. That is, for $k<k_a$, $x_i=0$ for all $i$ at optimality and all rides are served by AVs. The third column, labeled $k_s$, is the value of $k$ for which the optimal mass of AVs becomes zero; that is, for $k< k_s$, the optimal profit for the mixed autonomy system is strictly greater than the human-only system. 

From the numerical study , we observe that the optimal profit for the network with AVs is the same as that of the human-only network when $k$ is small, but exceeds the latter when $k$ increases. In addition, as established by Proposition \ref{prop:lowerbound}, $k_s\leq k_t$ for all $\beta$.

\begin{table}[h]
\caption{Theoretical and Observed values of $k$ vs. $\beta$}
\label{k values}
\begin{tabular}{|c|c|c|c|}
    \hline
& $k_a$, it is optimal to& $k_s$, it is optimal to & $k_t=1-\beta$,  \\
	$\beta$    &have only AVs& have some AVs  & upper bound \\
&for $k< k_a$&for $k< k_s$&from Proposition \ref{prop:lowerbound}\\
    \hline
    0.5 & 0.4383 & 0.4992 & 0.5 \\
    \hline
    0.55 & 0.3916 & 0.4417 & 0.45\\
    \hline
    0.6 & 0.3458 & 0.3856 & 0.4\\
    \hline
    0.65 & 0.3008 & 0.3312 & 0.35\\
    \hline
    0.7 & 0.2564 & 0.2786 & 0.3\\
    \hline
    0.75 & 0.2187 & 0.2256 & 0.25\\
    \hline
    0.8 & 0.1799 & 0.18	& 0.2\\
    \hline
    0.85& 0.1387 & 0.1388 & 0.15 \\
    \hline
    0.9 & 0.0947 & 0.095 & 0.1\\
    \hline
    0.95 & 0.0486 & 0.0487 & 0.05\\
    \hline
\end{tabular}
\label{tab:tab1}
\end{table}
Moreover, we make a few additional observations:
$k_s$ is always strictly less than the theoretical upper bound $k_t = 1-\beta$. 
The appearance of human-drivers and the disappearance of the AVs is not concurrent. That is, $k_a<k_s$ so that, for $k_a<k<k_s$, both human drivers and AVs are needed to obtain the optimal profit for the mixed-autonomy platform.


\section{Conclusions}
\label{sec:conclusions}

We proposed a model for ride-sharing systems with mixed-autonomy and showed that equilibrium conditions can be computed efficiently by converting the original problem into an alternative convex program. 
We found that the optimal profit for the ride-sharing platform if  AVs are introduced into the fleet will be the same as that of the human-only network when $k$ is small, \emph{i.e.}, the ratio of  outside option earnings for drivers' lifetime to the cost for operating an AV for the equivalent time of drivers' lifetime is relatively low. In particular, in Proposition \ref{prop:lowerbound}, we showed that if the cost of operating an AV exceeds the expected compensation to a driver in the system, the platform will find it optimal to not use AVs, an intuitively appealing result.

Surprisingly, the case study illustrates that the platform may not necessarily find it optimal to use AVs even when the cost of operating an AV is less than the expected compensation to a driver in the system. Our future research will seek to fully characterize these numerical observations.

In addition, the model proposed and studied here includes a number of simplifying assumptions. For example, in reality, destinations are not equidistant and ride costs might then depend on destination. While these simplifying assumptions are important for illuminating fundamental properties of ride-sharing in a mixed autonomy setting, future work will consider relaxing these assumptions.



\bibliography{./citation/Library} 
\bibliographystyle{IEEEtran}

\end{document}